\documentclass{amsart}
\usepackage[T1]{fontenc}
\usepackage{amssymb,amsmath,graphicx,dsfont,amsthm,verbatim,psfrag,enumerate,listing} 
\usepackage[usenames]{color}
\usepackage[all]{xy}
\usepackage{mathtools}
\usepackage{pgf,tikz}
\usepackage{tikz-cd}
\usetikzlibrary{arrows}

\usepackage{mathrsfs}

\newtheorem{theorem}{Theorem}[section]
\newtheorem{lemma}[theorem]{Lemma}
\newtheorem{corollary}[theorem]{Corollary}

\theoremstyle{definition}
\newtheorem{definition}[theorem]{Definition}
\newtheorem{example}[theorem]{Example}

\theoremstyle{remark}
\newtheorem{remark}[theorem]{Remark}

\numberwithin{equation}{section}

\newif\ifpaper
\newif\ifpapertwo
\papertwofalse
\paperfalse

\newcommand{\R}{\mathbb{R}}
\newcommand{\ps}[2]{\langle #1 , #2 \rangle}
\newcommand{\norm}[1]{\| #1 \|}

\newcommand{\M}{\mathcal{M}}

\def\epsilon{\varepsilon}

\newcommand{\diam}{\operatorname{diam}}

\newcommand{\Co}{\mathcal{C}} 
\newcommand{\V}{\mathcal{V}}  
  
 \newcommand{\So}{\mathcal{S}}

\renewcommand{\mathring}[1]{#1^{\circ}}

\newcommand{\kernel}{K}
\newcommand{\cone}{\mathit \Gamma}

\DeclareFontEncoding{FMS}{}{}
\DeclareFontSubstitution{FMS}{futm}{m}{n}
\DeclareFontEncoding{FMX}{}{}
\DeclareFontSubstitution{FMX}{futm}{m}{n}
\DeclareSymbolFont{fouriersymbols}{FMS}{futm}{m}{n}
\DeclareSymbolFont{fourierlargesymbols}{FMX}{futm}{m}{n}
\DeclareMathDelimiter{\VERT}{\mathord}{fouriersymbols}{152}{fourierlargesymbols}{147}

\begin{document}
\title[The contractivity of multilinear maps]{The contractivity of cone--preserving multilinear~mappings}

\author{Antoine Gautier}
\address{Department Mathematics and Computer Science, Saarland University, Saarbr\"ucken, Germany}
\email{antoine.gautier@uni-saarland.de}
\thanks{The work of A.G. has been supported by the ERC starting grant ``NOLEPRO'' No 307793.}

\author{Francesco Tudisco}
\address{School of Mathematics, University of Edinburgh, UK}
\email{f.tudisco@strath.ac.uk}
\thanks{The work of F.T. was funded by  the European Union's Horizon 2020 research and
innovation programme under the Marie Sk\l odowska-Curie individual fellowship ``MAGNET'' No 744014.}


\date{\today}



\begin{abstract}
With the notion of mode-$j$ Birkhoff contraction ratio, we prove a multilinear version of the Birkhoff-Hopf and the Perron-Fronenius theorems, which   provide conditions on the existence and uniqueness of a solution to a large family of systems of nonlinear  equations of the type $f_i(x_1,\dots,x_\nu)= \lambda_i x_i$, being   $x_i$ and element of a cone $C_i$ in a Banach space $V_i$. We then consider a family of nonlinear integral operators $f_i$ with positive kernel, acting on product of spaces of continuous real valued functions. In this setting we provide an explicit formula for the mode-$j$ contraction ratio which is particularly relevant in practice as this type of operators play a central role in numerous models and applications.
\end{abstract}

\maketitle

\section{Introduction}
In this work we use the term ``Birkhoff-Hopf Theorem'' to refer to a number of results of different authors that overall show that a large class of cone-preserving linear operators behaves as contraction mappings with respect to the Hilbert projective metric. The original versions of this classical result were independently proved by Birkhoff \cite{birkhoff1} and Hopf \cite{hopf1}. Various authors have contributed afterwards to extend, generalize and sharpen the original theorems. A very partial list of contributors include Eveson and Nussmaum \cite{BHNB}, Liverani and Wojtkowski \cite{liverani1994generalization}, Lim \cite{lim2003birkhoff}. For a review, we refer to the monograph by Lemmens and Nussbaum \cite{NB} and, in particular, Appendix A therein. 

The Birkhoff-Hopf theorem and the ideas behind it arise and play an important role in a wide variety of problems. For example,  this result can be used   to prove  so-called  linear  and  nonlinear
 ergodic theorems in population dynamics \cite{cohen1979ergodic}, it has applications in control theory and concerning ordinary  and partial  differential equations, particularly   filtering theory  \cite{budhiraja1998robustness,nussbaum1994finsler}, it is an important  tool in problems concerned  with  rescaling  matrices  or non-negative  integral kernels, so-called $DAD$ theorems and Sinkhorn-Knopp algorithms \cite{borwein1994entropy,brualdi1974dad,knight2008sinkhorn}, it has been used to provide tight lower bounds for the logarithmic Sobolev constant of Markov chains~\cite{pmpaper},  and, in particular, the ideas behind the theorem play a central role in the Perron-Frobenius theory for linear and nonlinear mappings that leave a closed cone invariant \cite{babaei2018stochastic,Gaubert,multiPF,liverani1995decay,NB,rugh2010cones}.

In this work we extend the Birkhoff-Hopf theorem to multilinear mappings acting on finite  products of cones in Banach spaces. This extension allows us to prove a new Perron-Frobenius theorem, based on a Hilbert cone-metric and the concept of Lipschitz matrix. Due to the Banach fixed point theorem, the latter result provides new conditions for the existence and uniqueness of a solution for systems of nonlinear  equations of the type $f_i(x_1,\dots,x_\nu)=\lambda_i x_i$, $i=1,\dots,\nu$, with $x_i$ being an element of the Banach space $V_i$, and $f_i$ being a nonlinear cone-preserving mapping.

The paper is structured as follows: In the next section we recall the most general formulation of the Birkhoff-Hopf theorem we are aware of. In Section \ref{sec:multilinear_BH} we introduce multilinear and weakly multilinear maps and extend the Birkhof-Hopf theorem to these mappings. Then, in Section \ref{sec:multilinear_PF} we use that theorem to prove a new Perron-Frobenius theorem for cone-preserving multilinear maps. As for the linear setting, our new results have a wide range of applications. In particular, in Section \ref{sec:applications} we consider a class of nonlinear integral operators and provide an explicit formula for the contraction ratio of this type of mappings. This formula is particularly useful to provide conditions on the existence, uniqueness and computability of a solution to various systems of functional equations that arise in a number of diverse  contexts. Examples include  the integral equations considered in \cite{Bus73,bushell1986cayley},  the generalized Schr\"{o}dinger equation  discussed in  \cite{ruschendorf1998closedness}, the optimization of Kullback-Leibler divergence in optimal transport \cite{benamou2015iterative}, the matching problem for hypergraphs \cite{nguyen2017efficient}, various eigenvalue and rank-one approximation problems for tensors \cite{friedland2014number,tensorPF}, the optimization of multivariate polynomials \cite{zhou2012nonnegative} and the analysis of central components in complex networks \cite{arrigo2019multi,tudisco2017node}. 

\section{Cone theoretic background}\label{Hilbsection}
We start by recalling few useful and standard definitions and properties of cones. See e.g.\ \cite{BHNB,Nussb} for further details.
Let $V$ be a real vector space and let $C\subseteq V$. We say that $C$ is a (convex and pointed) cone 
if $C$ is closed and convex set such that $\alpha x\in C$ for any $x\in C$ and any $\alpha \geq 0$,  $C\cap (-C)=\{0\}$. 
Any such a cone $C\subseteq V$ induces a partial ordering on $V$ defined as
$$x \preceq_C y \quad \text{if and only if}\quad y-x\in C.$$
If additionally $(V,\norm{\cdot})$ is a  Banach space, and there exists a constant $\gamma>0$ such that $\norm{x}\leq \gamma\norm{y}$ whenever $0\preceq_C x \preceq_C y$, we say that $C$ is a normal cone in $(V,\norm{\cdot})$. 
For a lighter notation, in what follows we write $x\preceq y$ in place of $x\preceq_C y$ when there is no danger of confusion. 

Given $x\in C$ and $y\in V$, we say that $x$ dominates $y$ if there exist $\alpha,\beta \in \R$ such that 
$\alpha x\preceq y \preceq \beta  x$. 
If $x$ dominates $y$ and $x\neq 0$, then one can consider the quantity
\begin{align*}
M(x/y;C) = \inf\{\beta \in \R \colon x \preceq \beta y\}\, .
\end{align*}
For $x,y\in C \setminus\{0\}$, we say that $x$ is comparable to $y$, and write $x\sim_C y$, if $x$ dominates $y$ and $y$ dominates $x$. This defines an equivalence relation on $C$ and its equivalence classes are called the components of $C$. In particular, for $u\in C\setminus\{0\}$, we denote by $C_u$ the equivalence class
\begin{equation}\label{defCu}
	C_u= \{x\in C  \colon x \sim_C u\}.
\end{equation}
Note that $C_u$ is itself a cone and, as $0$ only dominates itself, we  have $C_0=\{0\}$. 

\begin{definition}[Infimum]\label{def:inf}
Let $\Omega\subseteq C$ be a bounded subset of $C$. We say that $u\in \Omega$ is the infimum of $\Omega$, in symbols $u=\inf \Omega$, if $u$ is the largest lower bound of $\Omega$, that is $u\preceq x$ for all $x\in \Omega$ and if $v\preceq x$ for all $x\in \Omega$, then $v\preceq u$.  
\end{definition}

\begin{definition}[Hilbert projective metric] 
Let $x,y\in C$ be such that $x\sim_C y$. The Hilbert projective distance $d_C(x,y)$ between $x$ and $y$ is defined as follows: Either $x=y=0$ and we set $d_C(0,0)=0$, or $x,y\neq 0$ and 
$$d_C(x,y)=\log\Big( M(x/y;C)\, M(y/x;C)\Big).$$
\end{definition}

It is well-known that, for any $u\in C$, $(C_u,d_C)$ is a pseudo-metric space. In particular, for every $x,y\in C_u$, it holds 
\begin{equation*}\label{projprop}
d_C(\alpha\,x,\beta\, y) = d(x,y)\ \text{ for all }\ \alpha,\beta>0.
\end{equation*}
In other words, $d_C$ is a metric on the space of rays in $C_u$. 
There is a rich literature on the completeness of $d_C$, we recall below a result that is particularly useful for us (see \cite{NB}, for example). 
 \begin{lemma}
 \label{NBcomplete}
 Let $C$ be a normal closed cone in a Banach space. For $u \in C \setminus\{0\}$, suppose that $\varphi\colon C_u \to (0,\infty)$ is continuous in the norm topology and homogeneous (i.e.\ $\varphi(\alpha x)=\alpha \varphi(x)$ for all $\alpha\geq 0$), and  let $S_\varphi= \{x\in C_u\colon \varphi(x)=1\}$. Then $(S_\varphi,d_C)$ is a complete metric space.
 \end{lemma}
Next, we recall the definition of projective diameter. 
\begin{definition}[Projective diameter] 
Let $V$ be a real vector space and $C\subseteq V$ a cone. For $S\subseteq C$, the projective diameter of $S$ with respect to $d_C$ is defined as
$$\diam(S;C)=\sup\{d_C(x,y)\colon x,y\in S\text{ and }x\sim_Cy\}.$$
\end{definition}

The projective diameter is the key of the Birkhoff-Hopf theorem. In fact, the latter result shows that the  best Lipschitz constant of a linear mapping between cones can be expressed in terms of the projective diameter of the image of the mapping and thus allows for a formula that characterizes the contraction ratio of linear maps. We recall the theorem below:  
\begin{theorem}[Birkhoff-Hopf]\label{thm:linear_BH}
Let $C$ be a cone in a real vector space $V$, $\cone$ be a cone in a real vector space $W$ and $f\colon V\to W$  a linear map  with $f(C)\subseteq \cone$. Let $\kappa(f;C,\cone)$ be the smallest Lipschitz constant of $f\colon (C,d_C)\to (\cone,d_\cone)$, i.e.
\begin{align*} 
\kappa(f;C,\cone)=\inf\big\{\lambda \geq 0 \ \big| & \ d_\cone(f(x),f(y))\leq \lambda\,  d_C(x,y), \\ 
&\forall x,y \in C \text{ such that } x\sim_C y\text{ and } f(x) \sim_\cone f(y)\big\}.
\end{align*}
Then
$\kappa(f;C,\cone)= \tanh\big(\tfrac{1}{4}\diam(f(C);\cone)\big)$,
with the convention $\tanh(\infty)=1$. 
\end{theorem}
Different proofs of the above theorem are known, see for instance \cite{birkhoff1,cavazos2003alternative,BHNB}. The quantity $\kappa(f;C,\cone)$ is usually referred to as the Birkhoff contraction ratio  and in the following we shall write $\kappa(f)$ in place of $\kappa(f;C,\cone)$ when there is no danger of confusion.
\begin{remark}\label{rmk:positive_linear_map}
An important example of cone is $C=\R^n_+$, consisting of all nonnegative vectors in $V=\R^n$, with interior being the set of positive vectors $\R^n_{++}$. For this cone we have 
$d_{\R^n_+}(x,y) = \log\big( \max_i(x_i/y_i)\max_j(y_j/x_j)\big)$, for any $x,y\in \R^n_{++}$. Moreover, if $f:\R^n\to\R^n$ is linear and positive, that is $f(\R^n_+\setminus\{0\})\subseteq \R^n_{++}$, then  it holds (see  \cite[Thm. A.6.2]{NB} e.g.) 
\begin{equation}\label{matrixBHformula}
\diam(f(\R^n_{+});\R^n_{+})=\max_{i,j,k,l}\log\left(\frac{\ps{e_i}{f(e_j)}\ps{e_k}{f(e_l)}}{\ps{e_i}{f(e_l)}\ps{e_k}{f(e_j)}}\right)\, ,
\end{equation}
where $e_1,\dots,e_n$ is the canonical basis of $\R^n$ and $\ps{\cdot}{\cdot}$ the Euclidean scalar product. Therefore $\diam(f(\R^n_{+});\R^n_{+})$ is bounded and thus $\kappa(f)<1$, by Theorem \ref{thm:linear_BH}. In other words, any positive linear map $f\colon\R^n\to \R^n$ acts as a strict contraction with respect to the Hilbert metric. 

\end{remark}
\section{Birkhoff-Hopf theorem for weakly multilinear maps}\label{sec:multilinear_BH}
We are interested in the case of maps defined on product of cones $\Co = C_1\times \cdots \times C_\nu$. 
As the product of cones is a cone itself, the Birkhoff-Hopf theorem can be applied to linear mappings acting on $\Co$ with not much difficulties. However, by exploiting the particular product structure of $\Co$ we can obtain a  tighter contraction ratio given in terms of what we call the \textit{mode-$i$ Birkhoff contraction ratio} of the mapping. Moreover, this allows us to generalize the theorem from linear to multilinear mappings. 
\subsection{The contraction ratio of multilinear mappings}
Let $\nu$ be a positive integer. Let $V_1,\ldots,V_\nu$ and $W$ be real vectors spaces and define $\V = V_1\times \ldots \times V_\nu$, the Cartesian product of the $V_i$. For $x\in \V$ and $i=1,\ldots,\nu$, we denote by $x_i$ the projection of $x$ onto $V_i$ so that $x=(x_1,\ldots,x_\nu)$ with $x_j\in V_j$ for all $j$. 
\begin{definition}[Multilinear map]\label{def:multilinear_map}
Let $\V$ be as above. A mapping $f\colon \V\to W$ is multilinear if for each $i=1,\dots,\nu$ and every $z\in \V$, the mapping $f|_{z}^i\colon V_i\to W$ 
\begin{equation}\label{deffi}
	f|_{z}^i(x_i)=f(z_1,\ldots,z_{i-1},x_i,z_{i+1},\ldots,z_\nu) 
\end{equation}
is linear. 
\end{definition}

We aim at extending the Birkhoff-Hopf theorem to multilinear mappings. However, in order to facilitate its application, we introduce a wider class of maps, which we call \textit{weakly multilinear}, and prove the theorem for this larger class. 
\begin{definition}[Weak multilinearity]
	Let $f\colon \V\to W$. We say that $f$ is weakly multilinear, if there exist positive integers $s$ and $1\leq i_1<\ldots<i_s\leq \nu$ and a multilinear mapping $\tilde f\colon  V_{i_1}\times\ldots\times V_{i_s}\to W$ such that $f(x)=\tilde f(x_{i_1},\ldots,x_{i_s})$ for every $x\in\V$. 
\end{definition}
A simple example of mapping which is weakly multilinear but not multilinear is $f\colon V_1\times V_2\to W$ defined as $f(x)=g(x_1)$ where $g\colon V_1 \to W$ is a linear mapping. 

Let $C_1 \subseteq V_1,\ldots,C_\nu\subseteq V_{\nu}$ and $\cone\subseteq W$ be cones and let $\Co = C_1\times \ldots \times C_\nu$. Let $f\colon \V\to W$ be a weakly multilinear mapping such that $f(\Co)\subseteq \cone$. As for each fixed $z\in\Co$ and $i=1,\ldots,\nu$, the mapping $f|_z^i$ defined as in \eqref{deffi} is either linear or constant, $f|_z^i(C_i)$ is either a cone or a singleton. In both cases, $f|_z^i(C_i)$ is a subset of the cone $\cone$ and so we may consider its diameter with respect to $d_\cone$. This observation allows us to define the mode-$i$ Birkhoff contraction ratio of a weakly multilinear mappings as follows:
\begin{definition}[Mode-$i$ Birkhoff contraction ratio]\label{defKi}
	Let $f\colon \V\to W$ be a weakly multilinear map such that $f(\Co)\subseteq \cone$. The mode-$i$ Birkhoff contraction ratio of $f$ is defined as
	\begin{equation*}
	\kappa_{i}(f; \Co,\cone) =\sup_{z\in \Co} \tanh\big(\tfrac{1}{4}\diam(f|_z^i(C_i);\cone)\big).
	\end{equation*}
	For brevity, we write $\kappa_i(f)$ in place of $\kappa_{i}(f; \Co,\cone)$, when there is no danger of confusion.
\end{definition}
Note that if $f|_z^i$ is constant, then $f|_z^i(C_i)$ is a singleton and so~{$\diam(f|_z^i(C_i);\cone)=0$}, which implies that $\kappa_{i}(f)=0$.

Now, we state our Birkhoff-Hopf theorem for weakly multilinear mappings.
\begin{theorem}[Multilinear Birkhoff-Hopf]\label{mainres}
Let $f\colon  \V\to W$ be weakly multilinear and suppose that $f(\Co)\subseteq \cone$. Let $\kappa_{i}(f)$ be the mode-$i$ Birkhoff contraction ratio of $f$. Then it holds
$$
(\kappa_1(f),\dots, \kappa_\nu(f)) = \inf\Big\{\textstyle{\lambda \in \R^\nu_+ : d_\cone(f(x),f(y))\leq \sum_{i=1}^\nu \lambda_i\, d_{C_i}(x_i,y_i), \, \forall x,y\in\Co}\Big\}
$$
where the infimum is understood in the sense of Definition \ref{def:inf}. 
\end{theorem}
\begin{proof} 
	Set $\Co = C_1\times\ldots\times C_\nu$ and let $x,y\in \Co$ be such that $x\sim_{\Co}y$.
	First, suppose that $f$ is multilinear.
	Define $z[1],\ldots,z[\nu+1]\in \Co$ as $z[1]=x$, $z[\nu+1]=y$ and $z[i]=(y_1,\ldots,y_{i-1},x_{i},\ldots,x_\nu)$ for every $1<i\leq \nu$. By the triangle inequality, we~have 
	\begin{align*}\label{trieqpf}
	d_{\cone}\big(f(x),f(y)\big)
	\leq\sum_{i=1}^\nu d_\cone\big(f(z[i]),f( z[i+1])\big)=\sum_{i=1}^\nu d_\cone\big(f|^i_{z[i]}(x_i),f|^i_{z[i+1]}(y_i)\big).
	\end{align*}
	Since for any $i$ we have $f|^i_{z[i]}=f|^i_{z[i+1]}$,  by
	the Birkhoff-Hopf theorem \ref{thm:linear_BH} we deduce that 
	\begin{align*}
	d_{\cone}\big(f|^i_{z[i]}(x_i),f|^{i}_{z[i+1]}(y_i)\big)&\leq\tanh\big(\tfrac{1}{4}\diam(f|_{z[i]}^i(C_i);\cone)\big)\,d_{C_i}(x_i,y_i)\\
	&\leq\kappa_{i}(f)\,d_{C_i}(x_i,y_i).
	\end{align*}
	Combining the above inequalities proves the claim.
	Now, if $f$ is weakly multilinear, then there exist integers $1\leq i_1<\ldots<i_s\leq \nu$ and a multilinear mapping $\tilde f\colon V_{i_1}\times\ldots\times V_{i_s}\to W$ such that $f(v)=\tilde f(v_{i_1},\ldots,v_{i_{s}})$ for all $v\in\Co$. In particular, we then have 
	$d_\cone(f(x),f(y))=d_\cone(\tilde f(x_{i_1},\ldots,x_{i_{s}}),\tilde f(y_{i_1},\ldots,y_{i_{s}}))$. Hence, the above argument implies that
	\begin{equation}\label{tildeeq}
	d_\cone(f(x),f(y))\leq \sum_{j=1}^s \kappa_{j}(\tilde f)\, d_{C_{i_j}}(x_{i_j},y_{i_j}).\end{equation}
	Now, let $z\in\Co$. 
	As $f$ is weakly multilinear, $f|_z^i$ is either constant or linear. Let $i\in\{1,\ldots,\nu\}$ and suppose that there is $k$ such that $i=i_k$.
	Then, with $\tilde z=(z_{i_1},\ldots,z_{i_s})$, it holds
	\begin{align*}
	f|_z^i(C_i)&=f(z_{1},\ldots,z_{i-1},C_i,z_{i+1},\ldots,z_\nu)\\
	&=\tilde f(z_{i_1},\ldots,z_{i_{k-1}},C_i,z_{i_{k+1}},\ldots,z_{i_s})=\tilde f|_{\tilde z}^i(C_i).
	\end{align*}
	As the projection $z\mapsto \tilde z$ is surjective, it follows that $\kappa_{i_k}(f)=\kappa_{k}(\tilde f)$. Now, if $i\notin \{i_1,\ldots,i_s\}$, then it holds $\kappa_i(f)=0$. It follows that 
	$$ \sum_{j=1}^s \kappa_{j}(\tilde f)\, d_{C_{i_j}}(x_{i_j},y_{i_j})=\sum_{j=1}^s \kappa_{i_j}(f)\, d_{C_{i_j}}(x_{i_j},y_{i_j})=\sum_{i=1}^\nu \kappa_{i}(f)\, d_{C_{i}}(x_{i},y_{i}),$$
	and, with \eqref{tildeeq}, this shows that $(\kappa_1(f),\dots,\kappa_\nu(f))$ belongs to $\Omega = \{\lambda \in \R^\nu_+ : d_\cone(f(x),f(y))\leq \sum_{i=1}^\nu \lambda_i\, d_{C_i}(x_i,y_i), \, \forall x,y\in\Co\}$. 
	
	Finally note that, for any $\lambda \in \Omega$, any $z\in \Co$ and any $\varepsilon>0$, by Theorem \ref{thm:linear_BH}, there exist  $x^{\epsilon}_i,y^{\epsilon}_i\in C_i$ with $x^{\epsilon}_i\neq y^{\epsilon}_i$ such that 	$$\tanh\big(\tfrac{1}{4}\operatorname{diam}(f|_z^i(C_i);\Gamma)\big)\leq\frac{d_{\Gamma}(f|_z^i(x^{\epsilon}_i),f|_z^i(y^{\epsilon}_i))}{d_{C_i}(x^{\epsilon}_i,y^{\epsilon}_i)}+\epsilon\leq \lambda_i + \epsilon\, .$$
It follows that
$\kappa_i(f) \leq     \lambda_i+\epsilon$ and, as $\epsilon>0$ is arbitrary, we have $\kappa_i(f)\leq \lambda_i$. Hence, $(\kappa_1(f),\dots,\kappa_\nu(f))$ is a lower bound of $\Omega$. As $(\kappa_1(f),\dots,\kappa_\nu(f))\in\Omega$, it follows that $(\kappa_1(f),\dots,\kappa_\nu(f))$ is the infimum of $\Omega$ which concludes the proof. 
\end{proof}

In the next section, we   show how to exploit Theorem \ref{mainres} in order to obtain a  Perron-Frobenius theorem for multilinear self-mappings on~{$\Co = C_1\times\cdots\times C_\nu$}.

\section{Multilinear  Perron-Frobenius theorem}\label{sec:multilinear_PF}
As before, let $\V=V_1\times\cdots\times V_\nu$ be the product of Banach spaces and let $f_i:\V\to V_i$, $i=1,\dots,\nu$ be such that $f_i(\Co)\subseteq C_i$, for normal cones $C_1,\dots,C_\nu$ and $\Co = C_1\times\cdots\times C_\nu\subseteq \V$.  Our main goal in this section is to provide conditions under which the system of functional equations
\begin{equation}\label{eq:system_spectral_equations}
\left\{ 
\begin{array}{l}
f_1(x_1,\dots,x_\nu) = \lambda_1 x_1 \\
\,\,\vdots \\
f_\nu(x_1,\dots,x_\nu) = \lambda_\nu x_\nu
\end{array}
\right. \quad \lambda_1,\dots,\lambda_\nu \geq 0
\end{equation} 
has a unique solution, which can be efficiently computed. 

To this end, we need a preliminary fixed point result which we present in the setting of general metric spaces, as it does not depend on the choice of the metric.

\subsection{A Banach fixed point theorem on product of metric spaces}\label{fpsec}
Let $(M_1,\mu_1),\ldots,(M_\nu,\mu_\nu)$ be metric spaces and set $\M=M_1\times \dots \times M_{\nu}$. As above, if $x\in\M$, then we denote by $x_i$ the projection of $x$ onto $M_i$. Furthermore, consider the cone-metric $\delta:\M\times\M\to\R_+^\nu$ defined by
$$
(x,y)\mapsto \delta(x,y) = (\mu_1(x_1,y_1), \dots, \mu_\nu(x_\nu,y_\nu))\, .
$$

Note that the system of  equations \eqref{eq:system_spectral_equations} can be compactly written as $f(x) = (\lambda_1x_1, \dots,\lambda_\nu x_\nu)$, where $f$ is the self-map $f=(f_1,\dots,f_\nu):\V\to\V$. This suggests that existence of a solution to \eqref{eq:system_spectral_equations}   can be addressed by means of a fixed point argument. Moreover, if $f$ is a contraction, then the solution must be unique. In what follows we observe how the contractivity of $f$ can be related to the concept of Lipschitz matrix.
\begin{definition}
Let $(M_1,\mu_1),\ldots,(M_\nu,\mu_\nu)$ be metric spaces and let $\M=M_1\times \dots \times M_{\nu}$. 
We say that an entry-wise nonnegative matrix $A$ is a Lipschitz matrix for $f:\M\to\M$ if 
\begin{equation}\label{eq:lip-matrix}
\delta(f(x),f(y))\preceq A\, \delta(x,y) 
\end{equation}
for all 
$x,y\in \M$ and where the inequality is understood coordinate-wise, i.e.\ with respect to the partial ordering induced by $\R_+^\nu$.
\end{definition}

A Lipschitz matrix for  $f:\M\to\M$ gives an estimate on the variation of $f$ on each subspace $(M_i,\mu_i)$. However, the next theorem shows that the spectral radius of such matrix provides an estimate on the variation of $f$ on the whole space $\M$. Since the seminal work of Perov \cite{perov1964}, several authors have developed new fixed point results for vector valued and cone metric spaces (see for example \cite{cvetkovic2014quasi,radenovic2017some,vetro2018some} and the review \cite{aleksic2018new}) and our next result contributes to this active line of research. 
 In particular, note that such result extends the Banach fixed point theorem for cone-metrics (see e.g.\ \cite[Thm.\ 4.1]{aleksic2018new}) to the setting where a Lipschitz matrix exists, rather than just a Lipschitz constant with respect to some fixed metric product. 

 \begin{theorem}\label{thm:fixed-point}
 Let $(M_i,\mu_i)$ be complete metric spaces and let $\M = M_1\times\dots\times\M_\nu$. Let $f:\M\to\M$ be a mapping with Lipschitz matrix $A$. If $\rho(A)<1$, then $f$ has a unique fixed point $x^\star \in \M$ and $\lim_{n\to\infty}f^n(x) = x^\star$, for any $x\in \M$. Moreover, it holds
 $$
 \delta\big(f^n(x), x^\star\big) \preceq (I-A)^{-1}A^n \delta\big(f(x),x\big)
$$
for any $x\in \M$ and any integer $n\geq 0$.
 \end{theorem}
\begin{proof}
Let $x\in\M$, we have
$
\delta\big(f^2(x),f(x)\big)\preceq A\delta\big(f(x),x\big).
$ 
 Adding $\delta(f(x),x)$ on both sides of such inequality we get
\begin{equation*}
\delta\big(f^2(x),f(x)\big)+\delta(f(x),x)\preceq A\delta\big(f(x),x\big)+\delta(f(x),x),
\end{equation*}
which  can be rearranged into
\begin{equation*}
(I-A)\delta(f(x),x)\preceq \delta(f(x),x)-\delta\big(f^2(x),f(x)\big).
\end{equation*}
It follows from $\rho(A)<1$ that $I-A$ is invertible and that $(I-A)^{-1}$ is component-wise nonnegative. Therefore
\begin{equation*}
\delta(f(x),x)\preceq (I-A)^{-1}\Big( \delta(f(x),x)-\delta\big(f^2(x),f(x)\big)\Big).
\end{equation*}
Then, by the triangle inequality, for every $m\geq n>0$ we have
\begin{align}\label{mainineqthm1}
\delta\big(f^{m+1}(x),f^{n}(x)\big)&\preceq \sum_{i=n}^m\delta\big(f^{i+1}(x),f^{i}(x)\big)\notag \\
&\preceq (I-A)^{-1}\sum_{i=n}^m \delta\big(f^{i+1}(x)-f^i(x)\big)-\delta\big(f^{i+2}(x),f^{i+1}(x)\big)\notag \\
&= (I-A)^{-1} \Big( \delta\big(f^{n+1}(x),f^n(x)\big) - \delta\big(f^{m+2}(x),f^{m+1}(x)\big) \Big)
\end{align}
In particular, if we set $n=1$ and let $m\to\infty$, we get
\begin{equation}\label{eq:ser}
\sum_{i=1}^{\infty}\delta\big(f^i(x),f^{i+1}(x)\big) \preceq (I-A)^{-1}\delta\big(f^2(x),f(x)\big).
\end{equation}

Now, let $\|\cdot\|$ be any monotonic norm on $\R^\nu_+$, i.e.\ such that $a\preceq b$ implies $\|a\|\leq \|b\|$. Then $\vartheta(x,y) = \|\delta(x,y)\|$ is a metric on $\M$ and the topology induced by $\vartheta$ is the product topology. Thus, as $(M_i, \mu_i)$ are complete, $(\M,\vartheta)$ is complete too.
It follows from \eqref{eq:ser} that $\sum_{i=1}^\infty\vartheta(f^i(x),f^{i+1}(x))<\infty$ and so $\{f^{n}(x)\}_{n}$ is a Cauchy sequence. Thus, by completeness, $f^n(x)$ converges to some $x^\star\in \M$. As $f$ is continuous this implies $x^\star = \lim_n f^{n+1}(x) = f(\lim_n f^n(x))=f(x^\star)$, i.e. $x^\star$ is a fixed point of $f$. Moreover, $x^\star$ is the unique fixed point. In fact, if $y$ is such that $f(y)=y$, then 
$$
0 \leq \vartheta(x^\star,y) = \vartheta( f^n(x^\star), f^n(y)) \leq \|A^n\|\, \vartheta(x^\star,y)
$$ 
and letting $n\to\infty$ we have $x^\star = y$. Eventually, combined with  \eqref{mainineqthm1}, we have
\begin{align*}
\lim_{m\to\infty}\delta(f^n(x),f^{m+1}(x))=\delta\big(f^n(x),x^\star\big) &\preceq (I-A)^{-1}\delta\big(f^{n+1}(x),f^n(x)\big)\\
&\preceq (I-A)^{-1}A^n \delta(f(x),x)\, ,
\end{align*}
which concludes the proof.
\end{proof}

\begin{remark}\label{rem:multi-homogeneous}
 Multi-homogeneous mappings,  introduced in \cite{multiPF}, are an example of maps having a well-defined non-trivial Lipschitz matrix. Recall that, given a product of cones $\Co=C_1\times\cdots\times C_\nu$, a map $f:\Co\to \Co$ is said to be multi-homogeneous if there exist coefficients $B_{ij}\geq 0$ such that $f(x_1,\dots,\lambda x_j,\dots, x_\nu)_i = \lambda^{B_{ij}} f(x_1,\dots,x_\nu)_i$, for any nonnegative number $\lambda\geq 0$. The matrix $B$ is called \textit{homogeneity matrix} of $f$. When a multi-homogeneous map is order-preserving, i.e.\ when $x\preceq_{\Co} y$ implies $f(x)\preceq_{\Co} f(y)$, then its homogeneity matrix is also a Lipshitz matrix. In fact, as 
 $x = (x_1,\dots,x_\nu)\preceq_\Co (M(x_1/y_1;C_1)y_1,\dots,M(x_\nu/y_\nu;C_\nu)y_\nu) =: G(x/y)$ for any $x,y\in \Co$,  we have
\begin{align*}
d_{C_i}(f(x)_i,f(y)_i) &= \log\{M(f(x)_i/f(y)_i;C_i)\, M(f(y)_i/f(x)_i;C_i)\}\\
&\leq \log\{M(f(G(x/y))_i/f(y)_i; C_i)\, M(f(G(y/x))_i/f(x)_i; C_i)\}\\
&=\textstyle{\sum_{j=1}^\nu}B_{ij}\, d_{C_j}(x_j,y_j)\, . 
\end{align*}
Several examples of multi-homogeneous mappings can be found in \cite{multiPF}. In particular, note that any  multilinear map $f=(f_1,\dots,f_\nu):\Co\to\Co$ is multi-homogeneous, with Lipschitz matrix $B_{ij}=1$, for all $i,j$. However, Theorem \ref{mainres} shows that, for multilinear mappings, $A_{ij}=\kappa_j(f_i)$ is a Lipschitz matrix too and, since $\kappa_j(f_i)\leq 1$,  it always holds $\rho(A)\leq \rho(B)$.
\end{remark}
In the next section we combine Theorems \ref{mainres} and \ref{thm:fixed-point}  to obtain a Perron-Frobenius theorem for multilinear mappings.
\subsection{Perron-Frobenius theorem}\label{subsec:multilinear_PF}
Let $V_1, \ldots,V_\nu$ be Banach spaces, and let $C_1\subseteq V_1,\ldots,C_{\nu}\subseteq V_\nu$ be cones all different from $\{0\}$. Set $\Co=C_1\times \ldots \times C_{\nu}$ and let $u=(u_1,\ldots,u_\nu)$ with $u_i\in C_i\setminus\{0\} $ for all $i$. Define $\Co_u$ as $\Co_u=(C_1)_{u_1}\times \ldots \times (C_{\nu})_{u_{\nu}}$ where $(C_i)_{u_i}$ is the component of $u_i$ in $C_i$ defined in \eqref{defCu}. Finally, let $\varphi_i:(C_i)_{u_i}\to (0,\infty)$ be continuous and homogeneous, and consider the product of unit slices $\So(\Co_u)=S_{\varphi_1}\times \ldots \times S_{\varphi_\nu}$ where $S_{\varphi_i} =\{x_i\in (C_i)_{u_i} : \varphi_i({x_i})=1\}$ for every~$i \in \{1,\dots,\nu\}$.

Throughout this section we assume that each $C_i$ is normal and closed. Therefore, by Lemma \ref{NBcomplete}, we have that $(S_{\varphi_i}, d_{C_i})$ is a complete metric space for every $i=1,\dots,\nu$,  with $d_{C_i}$ being the Hilbert metric on $C_i$. Thus, when a cone-preserving map has a Lipshitz matrix $A$ with respect to the Hilbert metric and $\rho(A)<1$, the following Perron-Frobenius type theorem follows from Theorem~\ref{thm:fixed-point}
\begin{theorem}\label{mainPF}
	Let $u\in \V$ be such that $u_i\in C_i\setminus \{0\}$ for all $i$ and let $f\colon \Co_u\to \Co_u$ be such that $f(\Co_u)\subseteq \Co_u$. Let  $A$ be an entry-wise nonnegative matrix such that 
	\begin{equation*}
	d_{C_i}(f(x)_i,f(y)_i)\leq \sum_{j=1}^\nu A_{ij}\,d_{C_j}(x_j,y_j)\qquad\quad \forall x,y\in\M\, ,
	\end{equation*}	
	for all $i=1,\dots,\nu$.  If $\rho(A)<1$, then there exists a unique $x^\star\in \So(\Co_u)$ and unique positive coefficients $\lambda_1,\ldots,\lambda_\nu>0$ such that $f(x^\star)_i=\lambda_i\,x^\star_i$, for all $i=1,\dots,\nu$.  Moreover, the sequence
		 \begin{equation}\label{pmdef}
		x^{({n+1})} = \left(\textstyle{\dfrac{f(x^{(n)})_1}{\varphi_1\big(f(x^{(n)})_1\big)},\ldots,\dfrac{f(x^{(n)})_\nu}{\varphi_\nu\big(f(x^{(n)})_\nu\big)}}\right)
		\end{equation}
		converges to $x^\star$, as $n\to\infty$, for every $x^{(0)}\in \So(\Co_u)$ and, for every nonnegative $w$ such that $wA=\rho(A)w$, there exists $\gamma>0$ such that
	    \begin{equation}\label{eq:conv-bound}
	        w_i \, d_{C_i}\big(x^{(n+1)}_i,x^\star_i\big) \leq \gamma \, \rho(A)^n 
	    \end{equation}
		for all $i=1,\dots,\nu$ and any integer $n\geq 0$.
\end{theorem}
\begin{proof}
	Consider the map $g:\So(\Co_u)\to\So(\Co_u)$ defined by $g(x)_i = f(x)_i/\varphi_i(f(x)_i)$, for $i=1,\dots, \nu$. As $d_{C_i}$ is projective for every $i$, then $A$ is a Lipshitz matrix for $g$ and, by Theorem \ref{thm:fixed-point}, the sequence $x^{(n)}$ converges to the unique $x^\star \in \So(\Co_u)$ such that $g(x^\star)=x^\star$, for any $x^{(0)}\in\So(\Co_u)$. Moreover, there exist unique coefficients $\lambda_i>0$ such that $f(x^\star)_i = \lambda_i  x_i^\star$, for all  $i=1,\dots,\nu$. Finally, if $\delta:\So(\Co_u)\times \So(\Co_u)\to \R^\nu_+$ is the cone-metric  $\delta(x,y) = \big(d_{C_1}(x_1,y_1),\dots,d_{C_\nu}(x_\nu,y_\nu)\big)$, from  Theorem \ref{thm:fixed-point} we have
	\begin{align*}
        w\,\delta(x^{(n+1)},x^\star) &\leq w A^n(I-A)^{-1}\delta(x^{(1)},x^{(0)})\\
        &= \rho(A)^n (1-\rho(A))^{-1}  \big(w\,\delta(x^{(1)},x^{(0)})\big)
    \end{align*}
    which, with $\gamma = (1-\rho(A))^{-1}\sum_i w_i d_{C_i}(x^{(1)}_i,x^{(0)}_i)$, implies the final convergence bound. 
\end{proof}
\begin{remark}
Note that when $w_i=0$ for some $i$, in the theorem above, then \eqref{eq:conv-bound} does not give information on  the convergence rate of $d_{C_i}(x_i^{(n+1)},x^\star_i)$. In fact, the bound \eqref{eq:conv-bound} does not hold in general when $A$ has no positive left-eigenvector $w$ and a weaker upper-bound holds in this case:   
Combining Young's  theorem \cite{young1980norm} with Theorem \ref{thm:fixed-point} one easily deduces that there exists $\tilde \gamma>0$ such that 
$$d_{C_i}(x_i^{(n+1)},x^\star_i)\leq  \tilde \gamma\, \binom{n}{\nu-1}\rho(A)^{n-\nu+1}$$
for all $i=1,\dots,\nu$ and all $n\geq \nu$. 
\end{remark}

A direct application of Theorems \ref{mainres} and \ref{mainPF}  eventually leads to the following 
\begin{theorem}[Multilinear Perron-Frobenius]\label{thm:multilinear_PF}
	Let $u\in \V$ be such that $u_i\in C_i\setminus \{0\}$ for all $i$ and let $f=(f_1,\dots,f_\nu)\colon \V\to \V$ be such that $f(\Co_u)\subseteq \Co_u$ and such that, for every $i=1,\ldots,\nu$, the mapping $f_i\colon \V\to V_i$ is weakly multilinear. Then Theorem \ref{mainPF} holds for $f$, with $A_{ij}$ being
	$$A_{ij}= \kappa_j(f_i) \qquad \forall i,j=1,\ldots,\nu\, ,$$
	where $\kappa_j(f_i) = \kappa_j(f_i;\, \Co_u, (C_i)_{u_i})$ is the mode-$j$ Birkhoff contraction ratio of $f_i$. In particular, if $\rho(A)<1$, then the system of nonlinear equations \eqref{eq:system_spectral_equations} has a unique solution in $\So(\Co_u)$ which can be computed with the power sequence \eqref{pmdef}. 
\end{theorem}
\begin{proof}
	As $\Co_u$ and $(C_i)_{u_i}$ are normal cones and $(S_{\varphi_i},d_{C_i})$ are complete metric spaces by Lemma \ref{NBcomplete}, the result follows by applying Theorems \ref{mainres} and \ref{mainPF} to $f_i$ for every $i=1,\ldots,\nu$.
\end{proof}
\begin{remark}\label{rmk:contractivity_of_h}
The theorem above together with Theorem \ref{mainres}  extend the Birkhoff-Hopf and the Perron-Frobeius theorems to multilinear maps and obviously reduce to those classical results when $f:\V\to \V$ is linear and $\nu=1$. However, even for linear mappings, when $\nu>1$ our theorems  improve the contraction bound given by the standard Birkhoff-Hopf theorem. This is shown by the following simple example:

Consider the finite dimensional setting of Remark \ref{rmk:positive_linear_map}:  $C=\R^n_+$, $\mathring C=\R_{++}^n$, $V=\R^n$, $\Co=C\times C$ and  $\V = \R^n\times \R^n$. Let  $f:V\to V$ be a positive linear map, i.e.  $f(\R^n_+\setminus\{0\})\subseteq \R_{++}$. We know from the Birkhoff-Hopf theorem that $f$ acts as a contraction on $C$, i.e.\ $\kappa(f;C,C)=\kappa(f) < 1$ (see also Remark \ref{rmk:positive_linear_map}). Now, define $h:\V \to \V$ as $h(x_1,x_2) = (f(x_2),f(x_1))$. Then $h$ is still linear and  $h(\Co)\subseteq \Co$. Moreover,  
$$
A = \begin{pmatrix}
0 & \kappa_2(h_1)\\
\kappa_1(h_2) & 0
\end{pmatrix} = \begin{pmatrix}
0 & \kappa(f)\\
\kappa(f) & 0
\end{pmatrix}
$$
is a Lipschitz matrix for $h$. 
Noting that $\rho(A) = \kappa(f)$,  we deduce from Theorem \ref{thm:multilinear_PF} that   $h:\V\to\V$ acts as a contraction on $\Co$. However,  it is known that $\kappa(h; \Co, \Co)=1$   (see for instance \cite[Thm.~3.10]{Seneta1981}) and therefore the classical Birkhoff-Hopf result does not reveal any contractivity of $h$. 
\end{remark}
%
%
%
%
%
%
%
%
%
%
%
\section{Application to nonlinear integral operators}\label{sec:applications}
 In this section we consider a  class  of mappings $f_i$ defined as particular nonlinear integral operators acting on cones of continuous and nonnegative functions via a positive continuous kernel function $K$.   This setting generalizes the case of positive linear integral operators $f$,  originally considered  by Hopf, and extends the formula for $\kappa(f)$ known in that case \cite{BHNB,hopf1}.   

 In particular, for this kind of maps we obtain an explicit formula that gives an upper bound on $\kappa_j(f_i)$, the mode-$j$ contraction ratio of $f_i$,  in terms of the kernel function. Due to Theorem \ref{thm:main_integral_op}, this formula is useful to address existence, uniqueness and computability of a solution to  various systems of equations including for instance the integral equations considered in \cite{Bus73} and \cite{bushell1986cayley},  the generalized Schr\"{o}dinger equation discussed in \cite{ruschendorf1998closedness} and various eigenvalue  equations for hypermatrices (or tensors) \cite{friedland2014number,tensorPF} which are connected, for example, with optimal transport, hypergraph matching, network science  and multivariate polynomial optimization \cite{benamou2015iterative,nguyen2017efficient,tudisco2017node,zhou2012nonnegative}.

 Note that our goal is not to have the most general assumptions  possible, but to illustrate how to apply our results to this particular kind of mappings. For instance, few careful adjustments should be enough to transfer the results to spaces of integrable but possibly discontinuous functions.
%
%
%
%
%
%
%
%
%
\def\cont{ \mathscr C } 
\subsection{Systems of positive integral equations}\label{sec:integrals}
Let $X_1,\dots,X_\nu$ be compact Hausdorff spaces and let $\eta_1,\dots,\eta_\nu$ be regular Borel measures on $X_1,\dots,X_\nu$, respectively, i.e.\ such that $\eta_i(U_i)>0$ for any nonempty open subset $U_i\subseteq X_i$. For $i=1,\ldots,\nu$, let $V_i =\cont(X_i, \R)$ be the Banach space of continuous and real valued functions on~$X_i$.  
For $i=1,\dots,\nu$, let $C_i\subseteq V_i$  be the cone 
$$C_i = \{x_i\in V_i :x_i(\xi_i)\geq 0\,,  \forall \xi_i\in X_i\} \,  $$
and consider $\varphi_i:\mathring C_i\to(0,\infty)$ continuous and homogeneous. Note that the interior $\mathring{C}_i$ is the set of positive valued functions $\mathring C_i=\{x_i\in C_i : x_i(\xi_i)>0\}$. Moreover note that each $x_i\in V_i$ attains its maximum and its minimum, since each $X_i$ is compact and any $x_i\in V_i$ is continuous. It follows that every $x_i\in \mathring C_i$ is comparable with the constant function $\mathbf 1(\xi_i)=1$ and we have $\mathring C_i = (C_i)_{\mathbf 1}$, i.e.\ $\mathring C_i$ is the component of $C_i$ containing $\mathbf 1$. Thus, by Lemma \ref{NBcomplete}, we deduce that $(S_{\varphi_i},d_{C_i})$ is a complete metric space, with $S_{\varphi_i} = \{x_i\in \mathring C_i  : \varphi_i(x_i)=1\}$. 
Finally, set $\V=V_1\times \cdots \times V_{\nu}$,  $\Co = C_1\times \cdots \times C_\nu$, $\mathring \Co = \mathring C_1\times \cdots \times \mathring C_\nu$ and $\So(\mathring \Co)=S_{\varphi_1}\times \cdots \times S_{\varphi_\nu}$.

Now, for any $i$, let 
$\Omega_i = X_1\times \dots \times X_{i-1}\times X_{i+1}\times \dots \times X_{\nu}$
and consider a positive continuous kernel $\kernel:X_1\times\cdots\times X_\nu\to (0,\infty)$. Given real numbers $\alpha_{ij}\in \R$, $i,j\in \{1,\dots,\nu\}$, define the integral operator $f_i:\V\to V_i$ as
\begin{equation}\label{eq:integral_operator}
f_i(x)(\xi_i)= \int_{\Omega_i}  \kernel(\xi_1,\ldots,\xi_\nu)\prod_{\substack{j=1, j \neq i}}^{\nu}x_j(\xi_j)^{\alpha_{ij}}\,\,  d\,  \eta_j(\xi_j)\, ,
\end{equation}
for every $\xi_i\in X_i$. 

In the next Theorem \ref{thm:main_integral_op} we specialize Theorems \ref{mainres} and \ref{mainPF} to this setting and, in particular,  we derive an upper bound for the mode-$j$ contraction ratio of such $f_i$, in terms of $K$ and the coefficients $\alpha_{ij}$. In practice, this formula can be used to derive conditions for the existence and uniqueness of a solution to the following system of nonlinear integral equations
\begin{equation}\label{integraleq}
	f_i(x) =\lambda_i \, x_i ^{\,\gamma_i} \qquad \forall i =1,\ldots,\nu\, ,
\end{equation}
where $\gamma_1,\dots,\gamma_\nu$ are nonzero real exponents. Particular cases of this system of equations include for example the integral equations considered in \cite{Bus73} and \cite{bushell1986cayley}, when $\gamma_i=\alpha_{ij}=1$ for all $i,j$ and  the generalized Schr\"{o}dinger equation \cite{ruschendorf1998closedness}, when $\gamma_i=-1$ and $\alpha_{ij}=1$ for all $i,j$. When the spaces $X_i$ are discrete, other examples  include for instance the higher-oder Kullback-Leibler divergence problem \cite{benamou2015iterative}, the $(\sigma,p)$ tensor eigenvalue problem \cite{tensorPF}, the optimization of multivariate polynomials \cite{zhou2012nonnegative} and the best rank-one approximation of tensors \cite{friedland2014number}. Moreover, we will briefly discuss an example discrete problem in Example \ref{ex:discrete}. 
For $j = 1,\dots,\nu$, consider the following kernel cross-ratio
\begin{equation}\label{eq:contraction_integral_kernel}
\triangle_j(\kernel)=\max_{\substack{\xi_1\in X_1,\dots, \xi_\nu\in X_\nu\\ \xi_1'\in X_1,\dots, \xi_\nu'\in X_\nu}} \frac{\kernel(\xi_1,\dots,\xi_\nu)\, \kernel(\xi_1',\dots,\xi_\nu')}{\kernel(\xi_1,\dots,\xi_j', \dots,\xi_\nu)\, \kernel(\xi_1',\dots,\xi_j,\dots,\xi_\nu')} \, .
\end{equation}
Note that, since $\kernel$ is continuous, the maximum is attained. Moreover, since $\kernel$ is positive valued, it holds $f_i(C_i\setminus\{0\})\subseteq \mathring C_i$ for any $i=1,\dots,\nu$. We have

\begin{lemma}\label{lem:formula_for_integral_operators}
Let $f_i:\V\to V_i$ be defined as in \eqref{eq:integral_operator} with $\alpha_{ij}=1$, for all $i,j=1,\dots,\nu$. Then $f_i$ is weakly multilinear and for any $j=1,\dots,\nu$, $j\neq i$, it holds 
$$
\kappa_i(f_i)=0 \qquad \text{and} \qquad \kappa_j(f_i) \leq \tanh\Big(\frac 1 4 \log \triangle_{j}(\kernel) \Big)\, 
$$
with the convention that $\tanh(\infty)=1$. 
\end{lemma}
\begin{proof}
For a fixed $z\in\Co$, consider the map $f_i|_z^j:V_j\to V_i$ defined, as in \eqref{deffi}. 
From \eqref{eq:integral_operator} we see that $f_i|_z^i$ is constant, which implies $\kappa_i(f_i)=0$. Now, if $j\neq i$, for any $x_j\in C_j$ we have 
$$
d_{C_i}(f_i|_z^j(x_j),f_i|_z^j(y_j)) = \log\Big( M(f_i|_z^j(x_j)/f_i|_z^j(y_j); C_i) \cdot  M(f_i|_z^j(y_j)/f_i|_z^j(x_j); C_i) \Big)\, .
$$
Moreover, since $f_i|_z^j$ is continuous on the compact domain $X_i$, we have 
$$
e^{(d_{C_i}(f_i|_z^j(x_j),f_i|_z^j(y_j)))} = \max_{\xi_i,\xi_i'\in X_i}\frac{f_i|_z^j(x_j)(\xi_i) \, f_i|_z^j(y_j)(\xi_i')}{f_i|_z^j(y_j)(\xi_i)\, f_i|_z^j(x_j)(\xi_i')} =: \max_{\xi_i,\xi_i'\in X_i}R(x_j,y_j)(\xi_i,\xi_i')\, .
$$
Thus, if we let $\xi=(\xi_1,\dots,\xi_\nu)$ and $\xi'=(\xi_1',\dots,\xi_\nu')$, by definition we get
\begin{align*}
   R(x_j,y_j)(\xi_i,\xi_i') &=\frac {\int_{\Omega_i\times\Omega_i} \kernel(\xi)x_j(\xi_j)\prod_{l\neq i,j} z_l(\xi_l)\kernel(\xi')y_j(\xi_j')\prod_{l\neq i,j} z_l(\xi_l') \, d\eta_i^\times d\eta_i^\times} {\int_{\Omega_i\times\Omega_i} \kernel(\xi)y_j(\xi_j)\prod_{l\neq i,j} z_l(\xi_l)\kernel(\xi')x_j(\xi_j')\prod_{l\neq i,j} z_l(\xi_l') \, d\eta_i^\times d\eta_i^\times} \\
   &\leq \max_{\substack{\xi_k\in X_k, k\neq i\\\xi_k'\in X_k, k\neq i}} \frac { \kernel(\xi_1,\dots,\xi_\nu)x_j(\xi_j)\, \kernel(\xi_1',\dots,\xi_\nu')y_j(\xi_j')} { \kernel(\xi_1,\dots,\xi_\nu)y_j(\xi_j)\, \kernel(\xi_1',\dots,\xi_\nu')x_j(\xi_j')}\\
   &= \max_{\substack{\xi_k\in X_k, k\neq i\\\xi_k'\in X_k, k\neq i}} \frac{\kernel(\xi_1,\dots,\xi_\nu)\, \kernel(\xi_1',\dots,\xi_\nu')}{\kernel(\xi_1,\dots,\xi_j', \dots,\xi_\nu)\, \kernel(\xi_1',\dots,\xi_j,\dots,\xi_\nu')}
\end{align*}
Putting all together we obtain $ d_{C_i}(f_i|_z^j(x_j),f_i|_z^j(y_j)) \leq \log \triangle_j(\kernel)$, for any $x_j,y_j\in C_j$. Thus $\diam(f_i|_z^j(C_j);C_i)\leq \log \triangle_{j}(\kernel)$ for any $z\in \Co$,  and this concludes the proof. 
\end{proof}

Note that, since $\kernel$ is positive valued, if the coefficients $\alpha_{ij}$ are all nonnegative, then the map $f=(f_1,\dots,f_\nu):\Co\to\Co$, with $f_i$ as in \eqref{eq:integral_operator}, is  order-preserving and  multi-homogeneous,  with homogeneity matrix $B_{ij}=\alpha_{ij}$, $i,j=1,\dots,\nu$. As observed in Remark \ref{rem:multi-homogeneous}, this implies that $B$ is a Lipschitz matrix for $f$. However, similarly to the multilinear case, next Theorem \ref{thm:main_integral_op} exhibits another Lipschitz matrix $A$ for $f$, with $\rho(A)\leq \rho(B)$.

As before, let $\mathring \Co = \mathring C_1\times \cdots\times \mathring C_\nu$. We have 
\begin{theorem}\label{thm:main_integral_op}
Let $f_i$ be defined in terms of the kernel $\kernel$ as in \eqref{eq:integral_operator} and consider the  map $f=(f_1,\dots,f_\nu):\V\to\V$.  
Then Theorem \ref{mainPF} holds for $f$, with  $A_{ij}$ being 
$$
A_{ii}=0, \qquad A_{ij} = |\alpha_{ij}| \cdot \tanh\Big(\frac 1 4 \log \triangle_{j}(\kernel) \Big), \qquad \forall i,j=1,\dots,\nu\, . 
$$

Moreover, if $\gamma_i\neq 0$ for all $i=1,\dots,\nu$, let $B$ be the matrix with coefficients  $B_{ij}=A_{ij}/|\gamma_i|$. Then, if $\rho(B)<1$,  the system of nonlinear integral equations \eqref{integraleq} has a unique solution in $\So(\mathring \Co)$ which can be computed with the power sequence \eqref{pmdef}, applied to the scaled mappings $\tilde f(x)_i = f(x)_i^{1/\gamma_i}$, $i=1,\dots,\nu$.
\end{theorem}
\begin{proof}
For $i\in\{1,\dots,\nu\}$ consider the map $g_i:\V\to\V$ defined  by $g_i(x)_j(\xi_j) = x_j(\xi_j)^{\alpha_{ij}}$, for all $\xi_j\in X_j$ and all $j=1,\dots,\nu$. Note that $g_i(\mathring \Co)\subseteq \mathring \Co$ and that $f_i = \tilde f_i\circ g_i $, where $\tilde f_i$ is the weakly multilinear mapping defined as in \eqref{eq:integral_operator} but with $\alpha_{ij}=1$, for all $j=1,\dots,\nu$. Then, combining Lemma \ref{lem:formula_for_integral_operators} with Theorem \ref{mainres}, for any $x,y\in \V$ we have
\begin{align*}
d_{C_i}(f_i(x),f_i(y)) &= d_{C_i}(\tilde f_i\circ g_i (x), \tilde f_i\circ g_i(y)) \leq \sum_{j=1}^\nu \kappa_j(\tilde f_i)d_{C_j}(g_i(x)_j, g_i(y)_j)\\
&=\sum_{j=1}^\nu |\alpha_{ij}|\kappa_j(\tilde f_i) \, d_{C_j}(x_j,y_j) \leq \sum_{j=1}^\nu A_{ij}\, d_{C_j}(x_j,y_j)\, ,
\end{align*}
which shows that the assumptions of Theorem \ref{mainPF} hold for $f$ and $A_{ij}$. Moreover, for any $x,y\in \V$ we have 
$d_{C_i}(f_i(x)^{1/\gamma_i},f_i(y)^{1/\gamma_i})\leq \sum_{j=1}^\nu B_{ij}\, d_{C_j}(x_j,y_j)$. 
By taking the $1/\gamma_i$ power on both sides of \eqref{integraleq}, such inequality, together with Theorem \ref{mainPF}, implies the thesis.
\end{proof}
Our final result shows that, for the particular type of integral operator we are considering and a large range of choices of the real parameters $\alpha_{ij}$, the  system of nonlinear equations \eqref{integraleq} admits a solution with a global scaling coefficient, i.e.\  $\lambda=\lambda_i$, for all $i=1,\dots,\nu$. This is shown in the following
\begin{corollary}
Let $f$ and $B$ be defined as in Theorem \ref{thm:main_integral_op} and assume that $\alpha_{ij} = \alpha_{j}$, $\gamma_j\neq 0$  and $\alpha_j+\gamma_j\neq 0$, for all $i,j=1,\dots,\nu$. If $\rho(B)<1$, then there exist a unique  $\lambda >0$ and a unique $u\in \mathring \Co$ with $\int_{X_i}u_i^{\alpha_i+\gamma_i}\, d\eta_i=1$  such that $f_i(u)=\lambda\, u_i^{\gamma_i}$ holds for any $i=1,\dots,\nu$.
\end{corollary}
\begin{proof}
Since $\alpha_i+\gamma_i \neq  0$ by assumption, we can consider the continuous and homogeneous map  $\varphi_i(x_i) = (\int_{X_i}x_i^{\alpha_i+\gamma_i}d\eta_i)^{\frac 1 {\alpha_i+\gamma_i}}$.~Then, by Theorem \ref{thm:main_integral_op}, there exists a unique $u\in  \So(\mathring \Co)$ such that $f_i(u)(\xi_i)=\lambda_i u_i(\xi_i)^{\gamma_i}$ for all $\xi_i\in X_i$ and $i=1,\dots,\nu$. Multiplying by $u_i(\xi_i)^{\alpha_i}$ and integrating over $X_i$ both sides of such equations,  we~get
\begin{equation}\label{eq:lambdas}
\int_{X_i}f_i(u)(\xi_i)u_i^{\alpha_i}(\xi_i)\, d\eta_i(\xi_i) = \lambda_i \int_{X_i}u_i(\xi_i)^{\alpha_i+\gamma_i}\, d\eta_i(\xi_i)= \lambda_i\, ,
\end{equation}
where the rightmost identity holds because $u_i\in S_{\varphi_i} = \{x_i:\varphi_i(x_i)=1\}$. Noting that, by the assumption on the coefficients $\alpha_{ij}$, the left hand side of \eqref{eq:lambdas} does not depend on $i$ and is positive, we conclude.
\end{proof}
We conclude with an example application of Theorem \ref{thm:main_integral_op} to the problem of computing the norm of the Hilbert tensor.

\begin{example}\label{ex:discrete}
We consider here an example operator of the type \eqref{eq:integral_operator}, corresponding to the discrete kernel known as Hilbert tensor. In this example the $X_i$ are discrete finite spaces and, for simplicity, we further assume here  $X_1=\dots=X_\nu = \{1, \dots, n\}$. The general case can be analyzed in an analogous way.  In this setting we have $V_1 = \dots = V_\nu =  \R^{n}$,  $C_1=\dots =C_\nu = \R^{n}_+$ and the interior  $\mathring C_i=\R^{n}_{++}$ is the set of vectors with $n$ positive components. The Hilbert tensor is defined by 
$$
H(i_1, \dots, i_\nu)  = \frac{1}{i_1+\dots+i_\nu -\nu +1}
$$
for $i_1, \dots, i_\nu \in \{1, \dots, n\}$. In \cite{song2014infinite}, the following Hilbert--type inequality has been proved for $H$
$$
\sum_{i_1,\dots,i_\nu} \frac{|x(i_1)|\cdots|x(i_\nu)|}{i_1+\dots+i_\nu-\nu+1}\leq n^{\nu/2} \sin(\pi/n) \Big(\sum_{i=1}^n x(i)^2\Big)^{\nu/2} \, ,
$$
which boils down to the famous Hilbert inequality for linear operators when $\nu=2$ (see e.g.\ \cite{frazer1946note}). It is interesting to observe that such inequality provides an upper bound on the tensor norm
$$
\|H\|_{p_1, \dots, p_\nu} = \max\Big\{\, \big|\!\!\!\sum_{i_1,\dots, i_\nu} H(i_1,\dots,i_\nu)x_{1}(i_1)\cdots x_\nu(i_\nu)\big| \, : \|x_1\|_{p_1}=\cdots =\|x_\nu\|_{p_\nu} =1\Big\} 
$$
where $p_1, \dots, p_\nu \geq 1$. In fact, since $H$ is symmetric, when $p_1=\dots = p_\nu = 2$ the following tight inequality holds for all $x\in \R^n$ (see e.g.\ \cite{tensorPF})
$$
\sum_{i_1,\dots, i_\nu} H(i_1,\dots,i_\nu)x(i_1)\cdots x(i_\nu) \leq  \|H\|_{2,\dots,2} \, \, \|x\|_2^\nu  \, .
$$
However, unlike the matrix case, a good approximation of  $\|H\|_{2,\dots,2}$ is not always computable. New conditions that allow us to compute $\|H\|_{2,\dots,2}$ follow as a consequence of Theorem \ref{thm:main_integral_op}. In fact, if we look at the critical point conditions for $\|H\|_{p_1, \dots, p_\nu}$ we observe that a maximizer has to be a solution of the system of equations $f_i(x) = \lambda_i x_i$, $i=1, \dots, \nu$, where $f_i$ is of the form \eqref{eq:integral_operator}, with $K=H$ and $\alpha_{ij}=\frac{1}{p_i-1}$. Moreover, since $H$ is positive, any maximizer $x$ must be nonnegative, as otherwise we could replace $x$ with its absolute value $|x|$ and further increase the value of the objective function. As Theorem \ref{thm:main_integral_op} gives us conditions on the existence of a unique positive solution, this must be the global maximizer i.e.\ the one realizing the norm $\|H\|_{p_1,\dots,p_\nu}$ In addition, the particular structure of the  Hilbert kernel $H$ allows us to compute the cross-ratios \eqref{eq:contraction_integral_kernel} exactly. In fact, as $H$ is symmetric, we have $\triangle_1(H)=\dots=\triangle_\nu(H)=\triangle(H)$ and
$$
\triangle(H) = \frac{n^2 (\nu-1) + n(2-\nu)}{n\nu -\nu + 1} \, .
$$
This shows that, for example, when $p_1=\dots=p_\nu=2$, the spectral radius of the matrix $B$ in Theorem \ref{thm:main_integral_op} is smaller than one when $\nu=4$ and $n\leq 11$ or $\nu=3$ and $n\leq 13$. Thus, for these choices of $\nu$ and $n$, using the power sequence \eqref{pmdef}, we are guaranteed we can compute $\|H\|_{2,\dots,2}$ to an arbitrary precision. Figure \ref{fig:norms} compares the actual value of $\|H\|_{2,\dots,2}$ with the Hilbert--type bound $n^{\nu/2}\sin(\pi/n)$, for $\nu =2,3,4$, showing that the quality of the bound degrades when the order $\nu$ increases.
\begin{figure}[t]
\includegraphics[width=.9\textwidth]{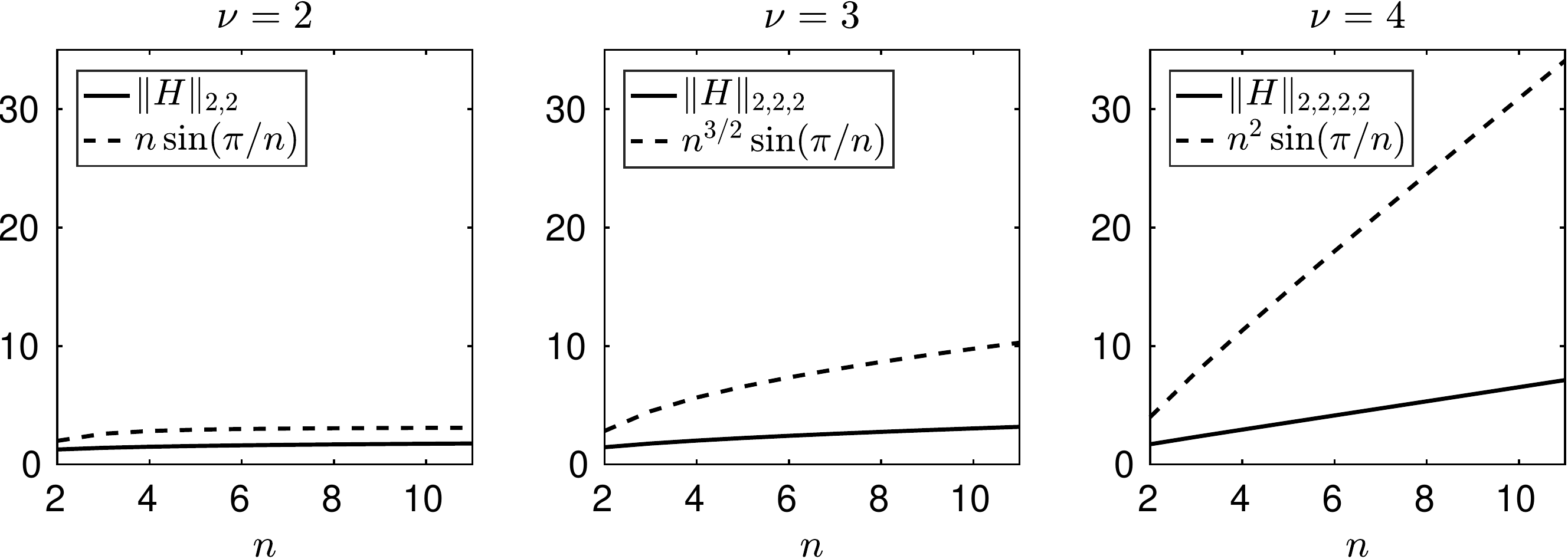}
\caption{This figure compares $\|H\|_{2,\dots,2}$ with its upper bound $n^{\nu/2}\sin(\pi/n)$ and shows that, while the bound is relatively tight in the linear case $\nu=2$, its precision degrades  when  $\nu$ increases.}\label{fig:norms}
\end{figure}

\end{example}

\providecommand{\bysame}{\leavevmode\hbox to3em{\hrulefill}\thinspace}
\providecommand{\MR}{\relax\ifhmode\unskip\space\fi MR }
\providecommand{\MRhref}[2]{%
  \href{http://www.ams.org/mathscinet-getitem?mr=#1}{#2}
}
\providecommand{\href}[2]{#2}

\end{document}